\documentclass[11pt]{amsart}
\usepackage{amsmath, amssymb, tensor, amsthm, amsfonts, marginnote, tikz-cd, mathtools, enumitem}
\usepackage[mathscr]{euscript}
\usepackage[hidelinks]{hyperref}
\newcommand\cO{\mathscr{O}}

\def\bC{\mathbf{C}}

\def\bF{\mathbf{F}}

\def\bP{\mathbf{P}}
\def\bQ{\mathbf{Q}}
\def\bR{\mathbf{R}}

\def\bV{\mathbf{V}}

\newcommand\frt{\mathfrak{t}}

\newcommand{\beq}{\begin{equation}}
\newcommand{\eeq}{\end{equation}}

\theoremstyle{plain}
\newtheorem{theorem}[subsection]{Theorem}
\newtheorem{proposition}[subsection]{Proposition}
\newtheorem{lemma}[subsection]{Lemma}
\newtheorem{corollary}[subsection]{Corollary}

\theoremstyle{remark}

\theoremstyle{definition}
\newtheorem{definition}[subsection]{Definition}

\DeclareMathOperator{\ad}{ad}

\DeclareMathOperator{\Spec}{Spec}
\DeclareMathOperator{\Frac}{Frac}

\DeclareMathOperator{\gr}{gr}

\DeclareMathOperator{\PGL}{PGL}
\DeclareMathOperator{\Hom}{Hom}
\DeclareMathOperator{\Ext}{Ext}
\DeclareMathOperator{\sct}{sc}
\DeclareMathOperator{\rt}{rt}

\DeclareMathOperator{\var}{var}
\DeclareMathOperator{\can}{can}
\DeclareMathOperator{\Mod}{Mod}
\DeclareMathOperator{\DMod}{DMod}
\DeclareMathOperator{\Bim}{Bim}
\DeclareMathOperator{\DBim}{DBim}
\DeclareMathOperator{\SBim}{SBim}
\DeclareMathOperator{\KSBim}{KSBim}

\DeclareMathOperator{\DK}{DK}
\DeclareMathOperator{\DShv}{DShv}
\DeclareMathOperator{\DSHV}{DShv}

\DeclareMathOperator{\Tilt}{Tilt}
\DeclareMathOperator{\KTilt}{KTilt}
\DeclareMathOperator{\aff}{aff}
\DeclareMathOperator{\SL}{SL}
\DeclareMathOperator{\Av}{Av}
\DeclareMathOperator{\Forg}{Res}

\newsavebox{\pullback}
\sbox\pullback{
\begin{tikzpicture}
\draw (0,0) -- (1ex,0ex);
\draw (1ex,0ex) -- (1ex,1ex);
\end{tikzpicture}}

\swapnumbers \numberwithin{equation}{section}

\title{Universal monodromic tilting sheaves}
\author{Jeremy Taylor}
\date{}

\begin{document}

\begin{abstract} 
We construct the universal monodromic big tilting sheaf on base affine space and calculate its endomorphisms. By formal completion, we recover Soergel's pro-unipotent Endomorphismensatz with arbitrary field coefficients. We give a Soergel bimodules description of the universal monodromic Hecke category and deduce a conjecture of Eberhardt that uncompletes Koszul duality.
\end{abstract}

\maketitle

\section{Introduction}

Let $G$ be a complex reductive group with connected center. Let $N \subset B$ be the unipotent radical of a Borel, and $T$ be a maximal torus. Let $Y \coloneqq G/N$ be base affine space, a $T$-torsor over the flag variety. Fix an arbitrary coefficient field $k$. Let $R \coloneqq k[\Lambda]$ be the group ring of the coweight lattice $\Lambda$.

The principal block of \cite{BGG75} category $\cO$ is equivalent to $B$-equivariant sheaves on $Y$. It admits the following deformation, constructed using sheaves with infinite dimensional stalks. Let $\DShv_{(B)}'(Y)$ be the weakly $B$-constructible (i.e. locally constant along the $B$-orbits) derived category of sheaves on $Y$. The universal monodromic Hecke category $\DShv_{(B)}(Y)$ is the full subcategory of compact objects. Left and right $T$-monodromy make it an $R$-bilinear category.


\subsection{Uncompleting Soergel's Endomorphismensatz}
We construct the universal monodromic big tilting sheaf $\Xi \in \DShv_{(B)}(Y)$, admitting standard and costandard filtrations and corepresenting certain vanishing cycles. Pushing forward to the flag variety recovers the non-monodromic \cite{BBM} tilting sheaf.

The big tilting sheaf is indecomposable, but splits completely after localizing to the regular locus in the dual torus $\check{T}$. We calculate its endomorphisms by localizing to the the regular and subregular locus then invoking Hartogs' lemma.

This is the same strategy as in \cite{So}, where Soergel interpolates between different blocks of category $\cO$ using modules on which the Cartan acts non-semisimply, but we work with sheaves rather than modules for the enveloping algebra. Under certain assumptions, \cite{GKM} explains how to reconstruct $\check{T}$-equivariant cohomology from just the 0 and 1 dimensional orbits, and our methods are Koszul dual to their equivariant localization.

\begin{theorem}\label{main}
Endomorphisms of the big tilting sheaf is $\Hom(\Xi, \Xi) \simeq R \otimes_{R^{W}} R$.
\end{theorem}

Our sheaf functors are all implicitly derived. Thus $\Hom(\Xi, \Xi)$ is a priori a complex of vector spaces. But it is concentrated in degree 0 by the standard and costandard filtrations.

Our arguments are logically independent of \cite{So, BY, BR}. 
Taking the fiber at $1 \in \check{T}$ recovers Soergel's theorem that endomorphisms of the non-monodromic big tilting sheaf equals $R \otimes_{R^W} k_1$. 
Formally completing Theorem \ref{main} recovers theorems of \cite{BR, Gou}.

\subsection{Uncompleting Soergel's Struktursatz and BGS Koszul duality}
We describe the universal monodromic Hecke category in terms of Soergel bimodules. This generalizes theorems of \cite{LY, Gou} by ecompassing all monodromies simultaneously.

Let $\check{X} \coloneqq \check{G}/\check{B}$ be the flag variety of the Langlands dual group. Eberhardt introduces the category $\DK_{\check{B}}(\check{X})$ of $\check{B}$-equivariant K-motives, an uncompletion of the $\check{B}$-equivariant derived category of sheaves. The main theorem of \cite{Eb} describes $\DK_{\check{B}}(\check{X})$ as the bounded homotopy category of multiplicative (also known as K-theory) Soergel bimodules.
We deduce the following conjecture of Eberhardt that uncompletes \cite{BGS} Koszul duality.
\begin{theorem}\label{Koszul}
There is an equivalence $\DShv_{(B)}(Y) \simeq \DK_{\check{B}}(\check{X})$.
\end{theorem}

K-motives pushed forward along Bott--Samelson resolutions correspond to tilting sheaves.

\subsection{Universal monodromic sheaves}

Let $j_w: Y_w \coloneqq B\dot{w} \hookrightarrow Y$ be the Borel orbit indexed by $w \in W$ in the Weyl group. There is a non-canonical isomorphism $Y_w \simeq \bC^{\ell(w)} \times T$. Let $R_{Y_w}$ be the universal local system on $Y_w$, the regular representation of the fundamental group. It is unique up to non-canonical isomorphism. Define the standard and costandard extensions \[\Delta_w \coloneqq j_{w!} R_{Y_w}[\ell(w) + \dim T] \quad \text{and} \quad \nabla_w \coloneqq j_{w*} R_{Y_w}[\ell(w) + \dim T].\]

Define the $R$-bimodule $R_w \coloneqq k[\Gamma_w]$ as functions on the graph $\Gamma_w \subset \check{T} \times \check{T}$ of $w$. The left and right monodromy actions on $\Delta_w$ and $\nabla_w$ differ by $w$. By adjunction \beq \label{StandardtoCostandard}\Hom(\Delta_w, \nabla_v) \simeq \begin{cases}R_w & \text{ if } w = v \\ 0 & \text{ otherwise.}\end{cases}\eeq

Constructing tilting sheaves is harder in the universal setting. We define $\Xi \coloneqq \Av_{(B)!} \chi_Y$ by averaging the Whittaker sheaf as in \cite{IY, LNY}. We will show that $\Xi$ admits a standard filtration by a vanishing cycles calculation, and a costandard filtration by using the longest Weyl group element.


\subsection{Endomorphismensatz proof outline}
Bimonodromy factors through \beq \label{BimonIntro} R \otimes_{R^W} R \rightarrow \Hom(\Xi, \Xi),\eeq a map of free right $R$-modules which we seek to prove is an isomorphism.

Let $\beta$ always denote a coroot and $t \in W$ the corresponding reflection. The dual torus $\check{T} \coloneqq \Spec R$ is stratified by intersections of walls $\check{T}_{\beta} \coloneqq \ker \beta$. By Hartogs' lemma it suffices to localize away from all higher codimension strata where multiple walls meet, then check that \eqref{BimonIntro} is an isomorphism on an open cover.

Let $\alpha$ always denote a simple coroot and $s \in W$ the corresponding simple reflection. The simple reflection tilting sheaf $\Xi_s$  supported on $\overline{Y}_s$, has endomorphisms $R \otimes_{R^s} R$. After right localizing away from all walls except $\check{T}_{\alpha}$, denoted by the $(\alpha)$ superscript, \[R \otimes_{R^W} R^{(\alpha)} \simeq \prod_{W/\langle s \rangle} R_w \otimes_{R^s} R^{(\alpha)} \qquad \text{and} \qquad \Xi^{(\alpha)} \simeq \bigoplus_{W/ \langle s \rangle} \Delta_w^{(\alpha)} * \Xi^{(\alpha)}_s\] 
both split with summands indexed by $w \in W$ minimal length in their $s$ coset, and therefore \eqref{BimonIntro} becomes an isomorphism. The same holds for non-simple coroots, and Theorem \ref{main} follows by Hartogs' lemma.

\subsection{Acknowledgments}
I am grateful to David Nadler for sharing his ideas on universal monodromic sheaves and to Zhiwei Yun for finding an error in an earlier draft of this paper. I also thank Harrison Chen, Gurbir Dhillon, Ben Elias, and the anonymous referee for their comments. This work was partially supported by NSF grant DMS-1646385.

\section{Universal monodromic sheaves}\label{SheavesAppendix}
We work in the analytic topology and allow infinite dimensional stalks, using the sheaf theory of \cite{KS} and \cite{BL}. (Their boundedness assumptions can be removed using \cite{Spal}.) Universal monodromic sheaves are elementary to define compared to pro-unipotent sheaves or K-motives.

\subsection{Monodromy}
Here we explain the $R$-bilinear structure on $\DShv_{(B)}(Y)$, using that weak constructibility with respect to the stratification by $T$-orbits is equivalent to equivariance for the universal cover. This is roughly a rewording of section 2.1 of \cite{BR}, whose arguments also extend to our universal monodromic setting.

Write $\DSHV'(X)$ for the derived category of all sheaves on a complex analytic space $X$. If $T \curvearrowright X$, let $\DSHV_{(T)}'(X)$ be the full subcategory of weakly $T$-constructible sheaves. 

Let $k_{\frt}$ be the constant sheaf on $\frt$, the universal cover of $T$. Let $\tilde{a}: \frt \times X \rightarrow X$ be  the (non-algebraic) action map, and define \beq \label{AvT} \Av_{(T)!}:\DSHV'(X) \rightarrow \DSHV_{(T)}'(X), \qquad K \mapsto \tilde{a}_!(k_{\frt} \boxtimes K)[2 \dim T].\eeq 
Let $\DSHV_{\frt}'(X)$ be the $\frt$-equivariant derived category.

\begin{lemma}\label{tEquivariant}
There is an equivalence $\DSHV_{\frt}'(X) \simeq \DSHV_{(T)}'(X)$.
\end{lemma}
\begin{proof}
According to \cite{BL}, the forgetful functor \[\Forg: \DSHV_{\frt}'(X) \rightarrow \DSHV'(X)\] admits a left adjoint $\Av_{\frt!}$, such that $\Forg \Av_{\frt!} \simeq \Av_{(T)!}$.
Since $\frt$ is contractible, $\Forg$ is fully faithful by Theorem 3.7.3 of \cite{BL}.

If $K \in \DSHV_{(T)}'(X)$ then $k_{\frt} \boxtimes K$ is locally constant along the fibers of $\tilde{a}$. Since those fibers are contractible, \beq \label{AvtFullyFaithful} K \simeq \tilde{a}_!(k_{\frt} \boxtimes K)[2 \dim T] \simeq \Forg \Av_{\frt!}K.\eeq Therefore $\DSHV_{(T)}'(X)$ is the essential image of $\Forg$.
\end{proof}

The coweight lattice $\Lambda \subset \frt$ acts trivially on $X$, so it acts by automorphisms of the identity functor, making $\DSHV_{(T)}'(X)$ an $R$-linear category. Moreover pushforward and pullback along $T$-equivariant maps are $R$-linear functors.

Denote the equivalence between $R$-modules and local systems by \[\DMod(R) \simeq \DShv_{(T)}'(T), \qquad M \mapsto M_T.\] Let $a: T  \times X \rightarrow X$ be the action map. Let $M \in \DMod(R)$ and $K, K' \in \DShv_{(T)}'(X)$. Define \[M \otimes_R K  \coloneqq a_!(M_T \boxtimes K)[2\dim T].\]
\begin{lemma} There are isomorphisms
\beq \label{TensorHom} \Hom(M \otimes_R K, K') \simeq \Hom_R(M, \Hom(K, K')),\eeq
and, if either $K$ or $M$ is compact, \beq \label{HomTensor} M \otimes_R \Hom(K, K') \simeq \Hom(K, M \otimes_R K').\eeq
\end{lemma}
\begin{proof}
The functor $- \otimes_R K: \DMod(R) \rightarrow \DShv_{(T)}'(X)$ is $R$-linear, because $a$ is equivariant for $T \curvearrowright T \times X$ by multiplication on the first factor and $T \curvearrowright T$ by multiplication. Moreover $R \otimes_R K \simeq K$ by \eqref{AvtFullyFaithful}. Resolving $M$ by free modules gives the desired isomorphisms.
\end{proof}

\subsection{Compactness}
The following is similar to Proposition G.3.5 of \cite{AGKRRV}.

\begin{lemma}\label{Compact}
The category $\DSHV_{(B)}'(Y)$ is compactly generated. Moreover $K \in \DSHV_{(B)}'(Y)$ is compact if and only if its stalks are perfect as complexes of $R$-modules.
\end{lemma}
\begin{proof}
Lemma 8.4.7(ii) of \cite{KS} implies that the $!$-restriction functor $\Hom_{\DSHV_{(B)}'(Y)}(\Delta_w, -)$ is isomorphic to restriction to a sufficiently small open ball containing $\dot{w}$ followed by compactly supported cohomology. Since these functors are continuous, $\Delta_w$ is compact in the weakly $B$-constructible category.

If $K \in \DSHV'_{(B)}(Y)$ and $Y_w$ is open in its support, then $\Hom(\Delta_w, K) \not\simeq 0$. Therefore $\DSHV'_{(B)}(Y)$ is compactly generated.
Suppose that the stalks of $K$ are all perfect over $R$. Then $K$ is compact because it admits a finite Cousin filtration whose $w$-graded piece is a finite complex of $\Delta_w$.

Suppose that $K$ is compact. Let $i: Y_1 \hookrightarrow Y$ be the closed stratum and $j: Y- Y_1 \hookrightarrow Y$ its open complement. Then $i^*$ preserves compactness because its right adjoint $i_* \simeq i_!$ is continuous. Therefore the stalks of $K|_{Y_1}$ are perfect $R$-modules. 
Also $j_! K|_{Y - Y_1}$ is compact, because it fits into a triangle $j_! K|_{Y - Y_1} \rightarrow K \rightarrow i_* K|_{Y_1}$, and by the previous paragraph $i_*$ preserves compactness. Since $j_!$ is continuous and fully faithful, $K|_{Y - Y_1}$ is compact. An inductive argument shows that all stalks of $K$ are perfect over $R$.
\end{proof}

\subsection{Convolution}
Here we verify that the usual convolution formulas 
still hold in the universal monodromic setting.
If $K \in \DSHV'(Y)$ and $K' \in \DSHV_{(B)}'(Y)$, define \[K * K' \coloneqq m_!(K \widetilde{\boxtimes} K')[\dim T]\] by pushforward along the multiplication map $m: G \times^N Y \rightarrow Y$. See section 4.3 of \cite{BY} for more details. Although $m$ is not proper, for weakly $T$-constructible sheaves $m_* \simeq m_![\dim T]$.

\begin{proposition}\label{AddConvolve}
If $\ell(vw) = \ell(v) + \ell(w)$ then there are noncanonical isomorphisms $\Delta_v * \Delta_w \simeq \Delta_{vw}$ and $\nabla_v * \nabla_w \simeq \nabla_{vw}$. 
\end{proposition}
\begin{proof}
It suffices to assume $v = s$ is a simple reflection. Let $N_{\check{\alpha}}^-$ and $N_{\check{\alpha}}$ be the corresponding negative and positive simple root spaces. Let $k_{N_{\check{\alpha}}!}$ be the !-extension of the constant sheaf on $N_{\check{\alpha}} \dot{s} = \dot{s} N_{\check{\alpha}}^- \subset Y$. Then $R_{Y_s}[\dim T] \simeq \Av_{(T)!} k_{N_{\check{\alpha}}!}$ on $Y_s \simeq T \times N_{\check{\alpha}}$. 

Because $\ell(sw) = 1 + \ell(w)$, the convolution map restricts to an isomorphism $\dot{s} N_{\check{\alpha}}^- \times Y_w \xrightarrow{\sim} Y_{sw}$. Thus $\Delta_s * \Delta_w \simeq \Av_{(T)!} k_{N_{\check{\alpha}}!} * \Delta_w[1] \simeq \Av_{(T)!} \Delta_{sw} \simeq \Delta_{sw}$ by \eqref{AvtFullyFaithful}.
The proof for costandards is similar.
\end{proof}

\begin{proposition}\label{InverseConvolve} If $w \in W$ then $\Delta_w * \nabla_{w^{-1}} \simeq \nabla_{w^{-1}} * \Delta_w \simeq \Delta_1$.
\end{proposition}
\begin{proof}
By Proposition \ref{AddConvolve} it suffices to consider $w = s$ a simple reflection. As above $\Delta_s * \nabla_s \simeq m_!(k_{N_{\check{\alpha}}} \boxtimes \nabla_s)$ is pushed forward along the $N_{\check{\alpha}}^-$-torsor $m: \dot{s} N_{\check{\alpha}}^- \times \overline{Y}_s \rightarrow \overline{Y}_s$.

If $n \in N_{\check{\alpha}}^-$, then $(\Delta_s * \nabla_s)|_{\dot{s}n} \simeq \Gamma_c(\nabla_s|_{N_{\check{\alpha}}^-}) \simeq \Gamma_1(\nabla_s|_{N_{\check{\alpha}}^-}) \simeq 0$, by base change 
and Lemma \ref{Contraction}. Therefore $\Delta_s * \nabla_s$ is supported on $Y_1$.

By base change $(\Delta_s * \nabla_s)|_{Y_1} \simeq R_{Y_1}[\dim T]$ is the !-pushforward of the universal local system along the projection $N_{\check{\alpha}} \times T \rightarrow T$. Therefore $\Delta_s * \nabla_s \simeq \Delta_1$. 
\end{proof}

We used the following contraction principle (see Proposition 3.7.5 of \cite{KS}).

\begin{lemma}\label{Contraction} Let $\bR^{> 0}$ act linearly on a vector space $V$ with positive weights. If $K$ is a weakly $\bR^{> 0}$-constructible sheaf on $V$ then \begin{enumerate}
\item\label{Contraction1} $\Gamma(K) \simeq \Gamma(K|_0)$,
\item\label{Contraction2} $\Gamma_0(K) \simeq \Gamma_c(K)$.
\end{enumerate}
\end{lemma}
\begin{proof} 
If $U \subset V$ is a convex open neighborhood of $0$ then Corollary 3.7.3 of \cite{KS} implies $\Gamma(K) \simeq \Gamma(K|_U)$. The collection of such open neighborhoods is cofinal, which implies \eqref{Contraction1}.

Let $V^*$ denote the one-point compactification of $V$. Applying \eqref{Contraction1} to the !-extension to $V^* - 0$ of $j^*K$ gives $\Gamma_c(j_*j^* K) \simeq 0$.
Therefore the triangle $\Gamma_x(K) \rightarrow \Gamma_c(K) \rightarrow \Gamma_c(j_*j^* K)$, implies the desired \eqref{Contraction2}. \end{proof}

\subsection{Perversity}
Perversity of non-monodromic standard and costandard sheaves is usually proved using Artin vanishing. But, following \cite{KS}, our perverse sheaves are allowed infinite dimensional stalks. Therefore Corollary 4.1.3 of \cite{BBD} does not directly apply in our setting. Below is an alternative argument using convolution.

\begin{proposition}\label{Perverse}
If $M \in \Mod(R)$ is finitely generated, then $M \otimes_R \Delta_w$ and $M \otimes_R \nabla_w$ are compact and perverse.
\end{proposition}
\begin{proof}
It suffices to prove that
\begin{enumerate}
\item \label{PerverseNabla} $\nabla_w \in \langle \Delta[\geq 0] \rangle$, the subcategory of $\DShv_{(B)}(Y)$ generated under extensions by $\Delta_v[i]$ for $v \in W$ and $i \geq 0$,
\item \label{PerverseDelta}$\Delta_w \in \langle \nabla[\leq 0] \rangle$, the subcategory of $\DShv_{(B)}(Y)$ generated under extensions by $\nabla_v[i]$ for $v \in W$ and $i \leq 0$.
\end{enumerate}
We will prove \eqref{PerverseNabla} by induction on $\ell(w)$, and \eqref{PerverseDelta} is similar. 
Choose a simple reflection $s$ satisfying $ws < w$. Then by induction $\nabla_{ws} \in \langle \Delta[\geq 0] \rangle$.

\begin{enumerate}
\item[-] If $vs < v$ then $\Delta_v * \nabla_s \simeq \Delta_{vs}$.

\item[-] If $v < vs$ then by \eqref{DeltaToNabla} there is a triangle $\Delta_{vs} \rightarrow \Delta_v * \nabla_s  \rightarrow \Delta_v/(e^{\alpha} - 1).$
\end{enumerate}

\noindent In both cases $\Delta_v * \nabla_s \in \langle \Delta[\geq 0] \rangle$. Hence also $\nabla_w \simeq \nabla_{ws} * \nabla_s \in \langle \Delta[\geq 0] \rangle$ as desired.
\end{proof}

\subsection{Left and right monodromy}
There exist no nontrivial degree 0 maps between different universal standard sheaves, because the $R$-bimodule structures are incompatible.

\begin{lemma}\label{MonodromyWeyl}
The left monodromy actions $R \curvearrowright \Delta_w, \nabla_w$ are obtained by twisting the right actions by $w$.
\end{lemma}
\begin{proof}
The projection map
\[p : Y_w \simeq B/(N \cap wBw^{-1})\dot{w} \rightarrow T\]
is equivariant for 
\begin{enumerate}
\item[-] the left action $T \curvearrowright Y_w$ and multiplication $T \curvearrowright T$,
\item[-] the right action $Y_w \curvearrowleft T$ and $w^{-1}$-twisted multiplication $T \curvearrowleft T$.
\end{enumerate}
Since $R_{Y_w}$ is the pullback under $p$ of the universal local system, the left and right $R$-actions on $R_{Y_w}$ differ by $w$. 
\end{proof}

The following is similar to Lemma 6.2 of \cite{BR}, and contrasts the non-monodromic setting.

\begin{proposition}\label{Hom0}
If $w \neq v$ then $\Hom^0(\Delta_w, \Delta_v) \simeq 0$ in degree 0.
\end{proposition}
\begin{proof}
Let $a: \Delta_w \rightarrow \Delta_v$. Choose $r \in R$ such that $r' \coloneqq w^{-1}(r) - v^{-1}(r) \neq 0$. Lemma \ref{MonodromyWeyl} implies that the right action of $r'$ on the image of $a$ is zero. 

However $\Delta_w/r'$ is perverse by Proposition \ref{Perverse}. Therefore $r'$ acts injectively on $\Delta_w$, and hence also on the image of $a$. Thus $a = 0$ as desired.
\end{proof}

\subsection{Intertwining functors}
Convolution with $\Delta_w$ induce derived autoequivalences that twist the $R$-bimodule structure as follows.

\begin{lemma}\label{Intertwining}
If $K, K' \in \DShv_{(B)}(Y)$, there is an isomorphism of bimodules \beq \label{IntertwiningIsomorphism} R_w \otimes_R \Hom(K , K')  \simeq \Hom(\Delta_w * K, \Delta_w * K').\eeq
\end{lemma}
\begin{proof}
Proposition \ref{InverseConvolve} implies that $\Delta_w * -$ is an equivalence of categories. Moreover Lemmas \ref{MonodromyWeyl} and \ref{MonodromyConvolve} imply that the left monodromy action $R \curvearrowright \Hom(\Delta_w * K, \Delta_w * K')$ coincides with the $w$-twisted left monodromy action $R \curvearrowright \Hom(K, K' )$. Therefore \eqref{IntertwiningIsomorphism} is an isomorphism of $R$-bimodules.
\end{proof}

\begin{lemma}\label{MonodromyConvolve}
If $K, K' \in \DShv_{(B)}(Y)$ then the right monodromy action $K \curvearrowleft R$ and the left monodromy action $R \curvearrowright K'$ induce the same action $R \curvearrowright K * K'$.
\end{lemma}
\begin{proof}
This follows because $m$ is equivariant for $T \curvearrowright G \times^N Y$ by $t(g, y) = (gt^{-1}, ty)$ and $T \curvearrowright Y$ trivially.
\end{proof}

\section{The big tilting sheaf} Here we construct the universal big tilting sheaf, then check that it admits standard and costandard filtrations.
There are three ways to construct the pro-unipotent big tilting sheaf:
\begin{enumerate}
\item Take an indecomposable summand of a Bott--Samelson tilting sheaf. In our universal setting it is not clear why this summand has the `correct' size (i.e. its pushforward to the flag variety is still indecomposable). Roughly for this reason Soergel works over a local ring in Theorem 6 of \cite{So}.
\smallskip

\item Use the \cite{BBM} tilting extension construction as in Lemma A.7.3 of \cite{BY} or Proposition 5.12 of \cite{BR}. Remark A.5.5 of \cite{BY} and Lemma 5.3 of \cite{BR} use that the formal completion of $R$ is a local ring, so it is unclear how to get a universal tilting extension of the `correct' size.
\smallskip

\item Start with the Whittaker sheaf supported on the open $B^-$ orbit. Then average it to be weakly $B$-constructible.
\end{enumerate}
We use the third approach for the reasons explained above. 
Our proof that $\Xi \coloneqq \Av_{(B)!} \chi_Y$ admits standard and costandard filtrations is different from Lemma 10.1 of \cite{BR}, because the universal big tilting sheaf is not the projective cover of a simple object.

\subsection{Whittaker averaging}
Here we construct the big tilting sheaf by averaging the Whittaker sheaf. It corepresents a certain vanishing cycles functor that we prove is monoidal.

Ionov and Yun use a similar construction in the pro-unipotent setting \cite{IY}. They define the Whittaker functional in terms of microlocalization and prove that it is monoidal. 

Let $\chi_{\bC}$, described in Equation (8.6.3) of \cite{KS}, be the weakly $\bC^{\times}$-constructible sheaf on $\bC$ that corepresents vanishing cycles. Beware that $\chi_{\bC}$ is not a character sheaf, unlike the Artin--Schreier sheaf or exponential D-module.

Let $f: N^- \rightarrow \bC$ be a character that vanishes on $[N^-, N^-]$ and is nontrivial on each negative simple root space. 
Let $\chi_{N^-} \coloneqq f^* \chi_{\bC}$ be the Whittaker sheaf on $N^-$. Let $\chi_{B^-} \coloneqq i_! \chi_{N^-}[-\dim T]$ be its extension along $i: N^- \rightarrow B^-$. Let $\chi_Y$ be its !-extension to $Y$ from the open $B^-$ orbit. 

\begin{definition}
Define the big tilting sheaf \[\Xi \coloneqq \Av_{(B)!} \chi_Y  \; \in \; \DSHV_{(B)}'(Y).\] Here $\Av_{(B)!}$, constructed similarly to \eqref{AvT}, is the left adjoint to forgetting weak $B$-constructibility. 
\end{definition}

\begin{definition} Define Soergel's functor \[\bV \coloneqq \Hom(\Xi, -): \DShv_{(B)}(Y) \rightarrow \DBim(R),\] taking values in the derived category of $R$-bimodules. 
\end{definition}

By Lemma \ref{Contraction}, Soergel's functor calculates $\phi_{f, 1}(i^!K|_{N^-})[\dim T]$, the stalk at $1 \in N^-$ of vanishing cycles for $f$. Indeed if $K \in \DShv_{(B)}(Y)$ then the sheaf $\underline{\Hom}(\chi_{N^-}, i^!K|_{N^-})$ is weakly constructible for the adjoint $\bC^{\times}$-action via the coweight $2\rho$.

\subsection{Monoidality}
According to \cite{LNY}, Soergel's functor is lax monoidal, and below we prove strictness. In the pro-unipotent setting, monoidality is proved in Lemma 4.6.4 of \cite{BY} using intersection cohomology sheaves.

\begin{proposition}\label{Monoidal}
For $K, K' \in \DShv_{(B)}(Y)$ we have $\bV(K) \otimes_R \bV(K') \simeq \bV(K * K')$.
\end{proposition}
\begin{proof}
Let $n: N^- \times N^- \rightarrow N^-$ and $b: B^- \times B^- \rightarrow B^-$ be multiplication. Let $\chi_{B^- \times B^-}$ be the extension to $B^- \times B^-$  of $n^* \chi_{N^-}[-2 \dim T]$.
Equation (3.7.12) and Exercise VIII.13 of \cite{KS} imply \beq \label{VLaxMonoidal}\bV(K) \otimes \bV(K') \simeq \Hom(\chi_{B^- \times B^-}, K|_{B^-} \boxtimes K'|_{B^-}).\eeq

Let $\Av_{(T)!} \chi_{B^- \times B^-}$ be obtained by averaging $\chi_{B^- \times B^-}$ to be weakly constructible for $T \curvearrowright B^- \times B^-$ by \[a: T \times B^- \times B^- \rightarrow B^- \times B^-, \qquad (t, b, b') \mapsto (bt^{-1}, tb').\] Tensoring with $k_1$, the augmentation $R$-module at $1 \in \check{T}$, gives \beq \label{Claim} k_1 \otimes_R \Av_{(T)!} \chi_{B^- \times B^-}[-\dim T] \simeq a_!(k_T \boxtimes \chi_{B^- \times B^-})[\dim T] \simeq b^* \chi_{B^-},\eeq 
by the following commuting diagram 
\[\begin{tikzcd}[cramped]
& B^- \times B^- \arrow[r, "b"] \ar[dr, phantom, "\usebox\pullback", very near start] & B^- \\
T \times B^- \times B^- \arrow[ur, "a"] \arrow[d, "\text{project}"'] \ar[dr, phantom, "\usebox\pullback", very near start]& \arrow[l, hook'] T \times N^- \times N^- \simeq b^{-1}(N^-) \arrow[u, hook] \arrow[r] \arrow[d] & N^- \arrow[u, hook] \\
B^- \times B^- & \arrow[l, hook'] N^- \times N^-.\arrow[ur, "n"'] & \\
\end{tikzcd}\]
\vspace{-.7cm}

\noindent Therefore Lemma \ref{OpenOpposite} implies \begin{align*}\bV(K) \otimes_R \bV(K') &\overset{\eqref{VLaxMonoidal}}{\simeq} \Hom_R(k_1, \Hom(\chi_{B^- \times B^-}, K|_{B^-} \boxtimes K'|_{B^-})[\dim T] \\ &\overset{\eqref{TensorHom}}{\simeq} \Hom_{B^- \times B^-}(k_1 \otimes_R \Av_{(T)!} \chi_{B^- \times B^-}, K|_{B^-} \boxtimes K'|_{B^-})[\dim T] \\ &\overset{\eqref{Claim}}{\simeq} \Hom_{B^- \times B^-}(b^* \chi_{B^-}, K|_{B^-} \boxtimes K'|_{B^-}) \\ &\overset{\eqref{OpenV}}{\simeq} \bV(K * K'). \qedhere\end{align*}
\end{proof}

Let $j_w^-:Y_w^- \coloneqq B^- \dot{w} \hookrightarrow Y$ be the opposite Borel orbit through $w \in W$. The following lemma allowed us to calculate $\bV(K * K')$ by restricting both factors to the open orbit $Y_1^- = B^-$.

\begin{lemma}\label{OpenOpposite}
If $K, K' \in \DShv_{(B)}(Y)$ then \beq \label{OpenV} \bV(K * K') \simeq \Hom(\chi_{B^-}, b_*(K|_{B^-} \boxtimes K'|_{B^-})).\eeq
\end{lemma}
\begin{proof}
For $w \neq 1 \in W$, some negative simple root space $N_{\check{\alpha}}^-$ acts trivially on $Y_w^-$. Because the convolution map is left $N_{\check{\alpha}}^-$ equivariant, $F \coloneqq ((j_w^-)_* (j_w^-)^! K) * K'$ is equivariant for the left action $N_{\check{\alpha}}^- \curvearrowright Y$. Therefore $f_*i^!F$ is equivariant for translation $\bC \curvearrowright \bC$ (i.e. is a constant sheaf) so it has no vanishing cycles \[\Hom(\chi_Y, ((j_w^-)_* (j_w^-)^! K) * K') \simeq \Hom(\chi_{N^-}, i^!F)[\dim T] \simeq \Hom(\chi_{\bC}, f_*i^!F)[\dim T] \simeq 0.\]
Hence there is an isomorphism \[\Hom(\chi_Y, K * K') \xrightarrow{\sim} \Hom(\chi_Y, ((j_1^-)_* K|_{B^-}) * K') \simeq \Hom(\chi_{B^-}, b_*(K|_{B^-} \boxtimes K'|_{B^-})).\qedhere\]
\end{proof}

\subsection{Simple reflection calculations}
Let $s$ be a simple reflection. On the corresponding negative simple root space, the Whittaker character induces an isomorphism $r:\bC \xrightarrow{\sim} N^-_{\check{\alpha}}$. There exists $y \in \bC$ and a lift $\dot{s} \in N(T)$ such that $r(1) N = \dot{s} r(y) N$.
Identify 
\begin{enumerate}[label=(\alph*)]
\item \label{BYs} $\bC \times T \xrightarrow{\sim} \overline{Y}_s \cap B^-$ by $(x, t) \mapsto r(x)tN$,
\item \label{Ys}$\bC \times T \xrightarrow{\sim} Y_s$ by $(x, t) \mapsto   \dot{s} r(yx) tN$.
\end{enumerate}
We claim that $\overline{Y}_s = P_s/N \rightarrow P_s/B$ is the $T$-torsor $\cO(-\alpha)$ over $\bP^1$. 

\begin{lemma} \label{Oalpha} The two trivializations differ by the transition function \[\bC^{\times} \times T \overset{\ref{BYs}}{\simeq} Y_s \cap B^- \overset{\ref{Ys}}{\simeq}  \bC^{\times} \times T, \qquad (z, t) \mapsto (z^{-1}, \alpha(z)t).\]
\end{lemma}
\begin{proof}
Any $z \in \bC^{\times}$ can be written $z = \check{\alpha}(t^{-1})$ for some $t \in T$. Therefore 
\[r(z)N 
= tr(1)t^{-1}N 
= t \dot{s} r(y) t^{-1}N
= \dot{s} ((\dot{s}^{-1} t \dot{s}) r(y) (\dot{s} t^{-1} \dot{s}^{-1}))(\dot{s} t \dot{s}^{-1} t^{-1})N 
= \dot{s} r(yz^{-1}) \alpha(z) N.\qedhere\]
\end{proof}

Let $\nabla_s^-$ denote the pullback of $\nabla_s$ along $\bC \times T \overset{\ref{BYs}}{\simeq} \overline{Y}_s \cap B^- \hookrightarrow Y$. We now calculate its vanishing cycles and stalk. 

\begin{lemma}\label{VanishingStalk}
Let $R_T$ be the universal local system on $T$. Then 
\begin{enumerate}
\item\label{Vanishing} $\Hom(\chi_{\bC} \boxtimes R_T, \nabla_s^-) \simeq R[\dim T]$,
\item\label{Stalk}$\Hom(k_{\bC} \boxtimes R_T, \nabla_s^-) \simeq R/(e^{\alpha} - 1)[\dim T]$.
\end{enumerate}
\end{lemma}
\begin{proof}
Let $\exp: \bC \rightarrow \bC$ be the exponential map, and $\delta$ be the skyscraper sheaf at $0 \in \bC$. Equation (8.6.6) of \cite{KS} gives a triangle \beq \label{VarTriangle} \exp_! k_{\bC}[1] \rightarrow \chi_{\bC} \rightarrow \delta.\eeq 
Therefore $\Hom(\delta \boxtimes R_T, \nabla_s^-) \simeq 0$ implies \beq\label{var} \Hom(\chi_{\bC} \boxtimes R_T, \nabla_s^-) \overset{\var}{\simeq} \Hom(\exp_! k_{\bC}[1] \boxtimes R_T, \nabla_s^-).\eeq

The following two observations imply $\exp^!(\nabla_s^-) \simeq k_{\bC}[1] \boxtimes R_T[\dim T]$, and hence \beq \label{NearbyRs} \Hom(\exp_! k_{\bC}[1] \boxtimes R_T, \nabla_s^-) \simeq \Hom(k_{\bC}[1] \boxtimes R_T, \exp^! (\nabla_s^-)) \simeq R[\dim T].\eeq
\begin{enumerate}
\item[-] Lemma \ref{Oalpha} implies $\nabla_s^-$ is the pullback of the universal local system along \beq \label{PullbackAlpha}\bC^{\times} \times T \overset{\ref{BYs}}{\simeq} Y_s \cap B^- \overset{\ref{Ys}}{\simeq} \bC^{\times} \times T \rightarrow T, \qquad (z, t) \mapsto \alpha(z)t.\eeq
\item[-] Lemma \ref{tEquivariant} implies that $R_T$ is equivariant for $\bC \xrightarrow{\exp} \bC^{\times} \curvearrowright
 T$ acting by the coweight $\alpha$. 
 \end{enumerate}
Combining \eqref{var} and \eqref{NearbyRs} implies Part \eqref{Vanishing}.

The triangle $k_{\bC} \rightarrow \chi_{\bC} \rightarrow \exp_! k_{\bC}[1]$ induces a triangle \[\Hom(\exp_! k_{\bC}[1] \boxtimes R_T, \nabla_s^-) \xrightarrow{\can} \Hom(\chi_{\bC} \boxtimes R_T, \nabla_s^-) \rightarrow  \Hom(k_{\bC} \boxtimes R_T, \nabla_s^-).\]
Equation \eqref{PullbackAlpha} implies that $\nabla_s^-$ is equivariant for $z \in \bC^{\times} \curvearrowright \bC \times T$ by \[\bC \times T \rightarrow \bC \times T, \qquad (x, t) \mapsto (zx, \alpha(z)^{-1} t).\] Therefore monodromy of nearby cycles acts by $e^{\alpha} \in R$ on \eqref{NearbyRs}. Equation (8.6.8) of \cite{KS} says that $\var \circ \can$ acts by $1 - e^{\alpha}$ on \eqref{NearbyRs}. Therefore \eqref{var} implies $\Hom(k_{\bC} \boxtimes R_T, \nabla_s^-) \simeq R/(e^{\alpha} - 1)[\dim T]$.
\end{proof}

\subsection{Vanishing cycles calculations}
Here we calculate vanishing cycles of standards and costandards, using the above simple reflection calculations and convolution.

\begin{proposition}\label{VStandard}
If $M \in \DMod(R)$ and $w \in W$ then
\begin{enumerate}
\item \label{VStandard1} $\bV(M \otimes_R \nabla_1) \simeq M$,
\item \label{VStandard2} $\bV(\Delta_w) \simeq \bV(\nabla_w) \simeq R_w$,
\item \label{VStandard3} $\nabla_w * \Xi  \simeq \Delta_w * \Xi  \simeq \Xi$.
\end{enumerate}
\end{proposition}
\begin{proof}
For \eqref{VStandard1}, the triangle \eqref{VarTriangle} and vanishing $\Hom(\exp_! k_{\bC}[1], \delta) \simeq 0$ imply \[\bV(M \otimes_R \nabla_1) \simeq \Hom_{\bC \times T}(\chi_{\bC} \boxtimes R_T, \delta \boxtimes M_T) \simeq \Hom_{\bC \times T}(\delta \boxtimes R_T, \delta \boxtimes M_T) \simeq M.\]
For \eqref{VStandard2}, it suffices by Proposition \ref{Monoidal} to assume that $w = s$ is a simple reflection. Lemma \ref{VanishingStalk}\eqref{Vanishing} says $\bV(\nabla_s) \simeq R_s$, and a similar argument shows $\bV(\Delta_s) \simeq R_s$.
For \eqref{VStandard3}, Lemma \ref{Intertwining} and Proposition \ref{Monoidal} give an isomorphism of functors \[\Hom(\nabla_w * \Xi, -) \simeq \Hom(\Xi, \Delta_{w^{-1}} * - ) \simeq \Hom(\Xi, -)\] from $\DShv_{(B)}(Y) \rightarrow \DMod(R)$. By the Yoneda lemma $\nabla_w * \Xi \simeq \Xi$, and similarly $\Delta_w * \Xi \simeq \Xi$.
\end{proof}

\subsection{Standard and costandard filtrations}
Here we show that $\Xi$ admits standard and costandard filtrations, in particular it is compact and perverse. The Whittaker construction was only needed to construct $\Xi$, and will not be used in the rest of the paper.

\begin{proposition}\label{XiCoStandard}
The big tilting sheaf $\Xi$ admits standard and costandard filtrations with each $\Delta_w$ and $\nabla_w$ appearing exactly once. 
\end{proposition}
\begin{proof}
The $w$-graded piece of the Cousin filtration on $\Xi$ is of the form $M \otimes_R \Delta_w$, for some $M \in \DMod(R)$. Propositions \ref{Monoidal} and \ref{VStandard} give an isomorphism of functors \[\Hom(M, -) \simeq \Hom(\Xi, - \otimes_R \nabla_w) \simeq \Hom(\Xi, - \otimes_R \nabla_1) \simeq \Hom(R, -)\] from $\DMod(R) \rightarrow \DMod(R)$. By the Yoneda lemma $M \simeq R$. Therefore $\Xi$ admits a standard filtration with each $\Delta_w$ appearing exactly once. 

Calculating $\Hom(\Delta_w, \Xi)$ appears more difficult because $\Xi$ is defined using left adjoints. But Proposition \ref{VStandard} says $\Xi \simeq \nabla_{w_0} * \Xi$, where $w_0 \in W$ is  the longest element. Since $\Xi$ admits a standard filtration, $\Xi \simeq \nabla_{w_0} * \Xi$ admits a costandard filtration with each $\nabla_{w_0v^{-1}} \simeq \nabla_{w_0} * \Delta_v$ appearing exactly once.
\end{proof}



\section{Calculations in semisimple rank one}\label{SmallTiltSec}
Here we construct the tilting sheaf $\Xi_s$ supported on $\overline{Y}_s$ and calculate its endomorphisms. Here $s$ is the simple reflection corresponding to a simple coroot $\alpha$.

\subsection{Construction} Here we construct the simple reflection tilting sheaf, and characterize it uniquely by its standard and costandard filtration.

\begin{proposition}\label{Rank1Stalk}
For $s$ a simple reflection,
\begin{enumerate}
\item \label{StalkNabla}$\nabla_s|_{Y_1} \simeq R_{Y_1}/(e^{\alpha} - 1)[\dim T]$,
\item $\Ext^1(\Delta_1, \Delta_s) \simeq \Ext^1(\nabla_s, \nabla_1) \simeq R/(e^{\alpha} - 1)$.
\end{enumerate}
\end{proposition}
\begin{proof} 
Part \eqref{StalkNabla} follows by Lemma \ref{VanishingStalk}\eqref{Stalk}.
It implies there is a short exact sequence \beq \label{DeltaToNabla} 0 \rightarrow \Delta_s \rightarrow \nabla_s \rightarrow \Delta_1/(e^{\alpha} - 1) \rightarrow 0.\eeq
Taking $\Hom(\Delta_1, -)$ shows \beq \label{ExtDelta}\Ext^1(\Delta_1, \Delta_s) \simeq \Hom^0(\Delta_1, \Delta_1/(e^{\alpha} - 1)) \simeq R/(e^{\alpha} - 1).\eeq 
Taking $\Hom(-, \nabla_1)$ shows \beq \label{ExtNabla}\Ext^1(\nabla_s, \nabla_1) \simeq \Ext^1(\Delta_1/(e^{\alpha} - 1), \nabla_1) \simeq R/(e^{\alpha} - 1). \qedhere\eeq 
\end{proof}


\begin{definition}
Define the simple reflection tilting sheaf \[\Xi_s \coloneqq \ker(\nabla_s \oplus \Delta_1 \rightarrow \Delta_1/(e^{\alpha} - 1)).\]
\end{definition}
\noindent By \eqref{ExtDelta} there is a standard filtration \beq \label{Ext} 0 \rightarrow \Delta_s \rightarrow \Xi_s \rightarrow \Delta_1 \rightarrow 0  \quad \text{classified by} \quad 1 \in R/(e^{\alpha} - 1) \simeq \Ext^1(\Delta_1, \Delta_s).\eeq  By \eqref{ExtNabla} there is a costandard filtration \[0 \rightarrow \nabla_1 \rightarrow \Xi_s \rightarrow \nabla_s \rightarrow 0  \quad \text{classified by} \quad 1 \in R/(e^{\alpha} - 1) \simeq \Ext^1(\nabla_s, \nabla_1).\]
The extension classes of the above sequences are defined up to scaling by a unit in $R/(e^{\alpha} - 1)^{\times}$. 


\begin{lemma} \label{XisUnique} If $K \in \DShv_{(B)}(Y)$ admits standard and costandard filtrations \beq \label{FilUnique} 0 \rightarrow \Delta_s \rightarrow K \rightarrow \Delta_1 \rightarrow 0 \quad  \text{and} \quad 0 \rightarrow \nabla_1 \rightarrow K \rightarrow \nabla_s \rightarrow 0\eeq then $K \simeq \Xi_s$ is the simple reflection tilting sheaf. 
\end{lemma}
\begin{proof}
Taking $\Hom(\Delta_1, -)$ into the standard filtration \eqref{Ext} gives an exact sequence \[\Hom^0(\Delta_1, \Delta_1) \xrightarrow{\delta} \Ext^1(\Delta_1, \Delta_s)  \rightarrow \Ext^1(\Delta_1, K),\] and the extension class of the standard filtration is classified by $\delta(1)$. By the costandard filtration, $\Ext^1(\Delta_1, K) \simeq 0$. Therefore $\delta(1)$ generates $\Ext^1(\Delta_1, \Delta_s) \simeq R/(e^{\alpha} - 1)$ and classifies $K \simeq \Xi_s$.
\end{proof}

\subsection{Properties}
The simple reflection tilting sheaf can be obtained by restricting the big tilting.

\begin{lemma}\label{XiRestrict}
There is an isomorphism $\Xi|_{\overline{Y}_s} \simeq \Xi_s$.
\end{lemma}
\begin{proof}
By Proposition \ref{VStandard} and adjunction, \[\Hom(\Xi|_{\overline{Y}_s}, \nabla_1) \simeq R_1 \quad \text{and} \quad \Hom(\Xi|_{\overline{Y}_s}, \nabla_s) \simeq R_s.\] Hence there is a standard filtration \beq \label{XiRestrictStandard} 0 \rightarrow \Delta_s \rightarrow \Xi|_{\overline{Y}_s} \rightarrow \Delta_1 \rightarrow 0.\eeq 

There is an isomorphism of functors \[\Hom(\Xi|_{\overline{Y}_s} * \nabla_s, -) \simeq \Hom(\Xi|_{\overline{Y}_s}, - * \Delta_s) \simeq \Hom(\Xi, - * \Delta_s) \simeq \Hom(\Xi, -) \simeq \Hom(\Xi|_{\overline{Y}_s}, -)\] from $\DShv_{(B)}(\overline{Y}_s) \rightarrow \DMod(R)$.
By the Yoneda lemma $\Xi|_{\overline{Y}_s} \simeq \Xi|_{\overline{Y}_s}*\nabla_s$. Convolving \eqref{XiRestrictStandard} by $\nabla_s$ gives a costandard filtration \[0 \rightarrow \nabla_1 \rightarrow \Xi|_{\overline{Y}_s} \rightarrow \nabla_s \rightarrow 0.\] Lemma \ref{XisUnique} implies the desired $\Xi|_{\overline{Y}_s} \simeq \Xi_s$.
\end{proof}

The following lemma will be needed in the proof Lemma \ref{BSSheafLoc}. 

\begin{lemma} \label{Rank1Square} The self convolution of a simple reflection tilting sheaf is $\Xi_s * \Xi_s \simeq \Xi_s \oplus \Xi_s$.\end{lemma}
\begin{proof}
Proposition \ref{InverseConvolve} says $\Delta_s*-$ is a derived equivalence. Therefore \[0 \rightarrow \Delta_s \rightarrow \Delta_s * \Xi_s \rightarrow \Delta_1 \rightarrow 0 \quad \text{is classified by} \quad 1 \in R/(e^{\alpha}-1) \simeq \Ext^1(\Delta_1, \Delta_s),\] hence $\Delta_s * \Xi_s \simeq \Xi_s$. 

Equation \eqref{StandardtoCostandard} implies $\Ext^1(\Xi_s, \Xi_s) \simeq 0$. Therefore the following splits \[0 \rightarrow \Delta_s * \Xi_s \rightarrow \Xi_s * \Xi_s \rightarrow \Delta_1 * \Xi_s \rightarrow 0,\] giving the desired $\Xi_s * \Xi_s \simeq \Xi_s \oplus \Xi_s$.
\end{proof}

\subsection{Endomorphisms}
Below we calculate endomorphisms of the simple reflection tilting sheaf. In Appendix C of \cite{BY}, Yun performs a similar calculation using Frobenius weights.

\begin{proposition}\label{SimpleEnd}
There are isomorphisms $\bV(\Xi_s) \simeq \Hom(\Xi_s, \Xi_s) \simeq R \otimes_{R^s} R$.
\end{proposition}
\begin{proof}
The kernel of $\Xi \rightarrow \Xi|_{\overline{Y}_s}$ is filtered by standards $\Delta_w$ indexed by $w \neq 1, s$. Therefore Lemma \ref{XiRestrict} implies $\bV(\Xi_s) \simeq \Hom(\Xi|_{\overline{Y}_s}, \Xi_s) \simeq \Hom(\Xi_s, \Xi_s)$.

Equation \eqref{StandardtoCostandard} implies $\Hom(\Xi_s, \Xi_s)$ is concentrated in degree 0. By adjunction $\Hom(\Delta_s, \Delta_1) \simeq 0$, so there is an $R$-bimodule map \[\gr: \Hom(\Xi_s, \Xi_s) \rightarrow R_1 \times R_s, \qquad a \mapsto (\gr_1(a), \gr_s(a))\] making the following diagram commute
\[\begin{tikzcd}[cramped]
0 \arrow[r] & \Delta_s \arrow[d, "\gr_s(a)"] \arrow[r] & \Xi_s \arrow[d, "a"] \arrow[r] & \Delta_1 \arrow[r] \arrow[d, "\gr_1(a)"]& 0 \\
0 \arrow[r] & \Delta_s \arrow[r] & \Xi_s \arrow[r] & \Delta_1 \arrow[r] & 0. \\
\end{tikzcd}\]
\vspace{-.7cm} 

\noindent Moreover $\gr$ is injective because Proposition \ref{Hom0} says $\Hom^0(\Delta_1, \Delta_s) \simeq 0$.

Let $a \in \Hom(\Xi_s, \Xi_s)$ and write $a = b + \gr_1(a)$. Then $\gr_1(b) = 0$ so $b$ factors through $\Delta_s$ as shown in the commuting diagram 
\[\begin{tikzcd}[cramped]  & \Delta_s \arrow[r] \arrow[d, "\gr_s(b)"'] & \Xi_s \arrow[d, "b"] \arrow[dl, "b'"', dashed] & & \\  0 \arrow[r] &\Delta_s \arrow[r] & \Xi_s \arrow[r] & \Delta_1 \arrow[r] & 0.\\ \end{tikzcd}\]
\vspace{-.7cm} 

\noindent Taking $\Hom(-, \Delta_s)$ into the standard filtration \eqref{Ext} gives an exact sequence \[\Hom^0(\Xi_s, \Delta_s) \xrightarrow{\gamma} \Hom^0(\Delta_s, \Delta_s) \xrightarrow{\delta} \Ext^1(\Delta_1, \Delta_s) \simeq R/(e^{\alpha} - 1) \rightarrow \Ext^1(\Xi_s, \Delta_s) \simeq 0.\] Since $\delta$ is surjective, $\gr_s (b) = \gamma (b')$ (the restriction to $\Delta_s$ of $b'$) vanishes on $\check{T}_{\alpha}$. Hence $\gr_s (a)$ and $\gr_1 (a)$ agree along $\check{T}_{\alpha}$.

Since $G$ has connected center, the proof of Lemma \ref{AdjointWalls} shows that $\check{T}_{\alpha}$ is the $s$ fixed locus in $\check{T}$. Therefore $R \otimes_{R^s} R \subset R_1 \times R_s$ consists of pairs of functions that agree along $\check{T}_{\alpha}$.  Hence $\gr$ factors through $\Hom(\Xi_s, \Xi_s) \rightarrow R \otimes_{R^s} R$, which is an isomorphism because $\gr$ is injective and $R \otimes_{R^s} R$ is a cyclic $R$-bimodule generated by $1 \otimes 1$.
\end{proof}

\section{Localizing universal sheaves}
Here we describe the universal Hecke category after localizing away from all but one wall.
If $\beta$ is a coroot, let $R^{(\beta)}$ be the ring of functions on $\check{T}^{(\beta)} \coloneqq \check{T} - \bigcup_{\beta' \neq \beta} \check{T}_{\beta'}$, obtained by localizing away from all walls except $\check{T}_{\beta}$. 
For $K \in \DShv_{(B)}(Y)$, write $K^{(\beta)} \coloneqq K \otimes_R R^{(\beta)}$ for its localization with respect to the right $R$-action.

\subsection{Cleanness and blocks} 
Roughly speaking, after the localizing away from all walls except $\check{T}_{\alpha}$, all extensions are clean in the base directions of $Y \rightarrow G/P_s$, the only nontrivial extensions are in the fiber directions.

Let $\beta$ be a coroot, and $t \in W$ be the corresponding reflection.

\begin{proposition}\label{Key}
If $\ell(w) < \ell(wt)$ then  $\Delta_w^{(\beta)} \simeq \nabla_w^{(\beta)}$, i.e. the extension is clean.
\end{proposition}
\begin{proof}
Induct on the length $\ell(w)$. Choose a simple reflection $s$ satisfying $ws < w$. Necessarily $s \neq t$ so localizing \eqref{DeltaToNabla} kills the cokernel, and $\Delta_s^{(\beta)} \simeq \nabla_s^{(\beta)}$ becomes an isomorphism. Moreover \[\ell(ws) = \ell(w) -1 < \ell(wt) - 1 \leq \ell(wts) = \ell((ws)(sts)).\]
By the inductive hypothesis $\Delta_{ws}^{(s\beta)} \simeq \nabla_{ws}^{(s\beta)}$ and therefore \[\Delta_w^{(\beta)} \simeq \Delta_{ws}^{(s\beta)}* \Delta_s^{(\beta)} \simeq \nabla_{ws}^{(s\beta)} * \nabla_s^{(\beta)} \simeq \nabla_w^{(\beta)}.\qedhere\] 
\end{proof}

The following argument is similar to 4.11 of \cite{LY}.

\begin{proposition}\label{LocalizedExtensions}
If $w \neq v, vt$ then $\Hom(\Delta_w^{(\beta)},  \Delta_v^{(\beta)}) = 0$.
\end{proposition}
\begin{proof}
Assume that $w \neq v, vt$, and induct on the length $\ell(w)$.  Choose a simple reflection $s$ satisfying $ws < w$. 

\begin{enumerate}
\item[-] If $s \neq t$, then $\Delta_s^{(\beta)} \simeq \nabla_s^{(\beta)}$ is clean. Either $vs < v$ or $v < vs$ but, since $\Delta_s^{(\beta)} \simeq \nabla_s^{(\beta)}$ is clean, in both cases $\Delta_{vs}^{(s\beta)} * \Delta_s^{(\beta)} \simeq \Delta_v^{(\beta)}$. Therefore by induction \[\Hom(\Delta_w^{(\beta)},  \Delta_v^{(\beta)}) \simeq \Hom(\Delta_{ws}^{(s\beta)} * \Delta_s^{(\beta)}, \Delta_{vs}^{(s\beta)} * \Delta_s^{(\beta)}) \simeq \Hom(\Delta_{ws}^{(s\beta)}, \Delta_{vs}^{(s\beta)}) \simeq 0.\]

\item[-] If $s = t$ and $vt < v$, then by induction \[\Hom(\Delta_w^{(\beta)}, \Delta_v^{(\beta)}) \simeq \Hom(\Delta_{wt}^{(\beta)} * \Delta_t^{(\beta)}, \Delta_{vt}^{(\beta)} * \Delta_t^{(\beta)}) \simeq \Hom(\Delta_{wt}^{(\beta)}, \Delta_{vt}^{(\beta)}) \simeq 0.\]

\item[-] If $s = t$ and $v < vt$, then by induction $\Hom(\Delta_{wt}^{(\beta)}/(e^{\beta} - 1), \Delta_v^{(\beta)}) \simeq 0$ and \[\Hom(\Delta_{wt}^{(\beta)} * \nabla_t^{(\beta)}, \Delta_v^{(\beta)}) \simeq \Hom(\Delta_{wt}^{(\beta)}, \Delta_v^{(\beta)} * \Delta_t^{(\beta)}) \simeq \Hom(\Delta_{wt}^{(\beta)}, \Delta_{vt}^{(\beta)}) \simeq 0.\] Convolving $\Delta_{wt}$ by \eqref{DeltaToNabla} gives a triangle \[\Delta_w^{(\beta)} \rightarrow \Delta_{wt}^{(\beta)} * \nabla_t^{(\beta)} \rightarrow \Delta_{wt}^{(\beta)}/(e^{\beta} - 1),\] and hence $\Hom(\Delta_w^{(\beta)}, \Delta_v^{(\beta)}) \simeq 0$. \qedhere
\end{enumerate}
\end{proof}

\subsection{Localizing the big tilting}
After localizing away from all walls except $\check{T}_{\alpha}$, only standards indexed by the same $s$ coset admit nontrivial extensions, hence the following splitting.

\begin{proposition}\label{TiltingSplitting}
The localized big tilting sheaf splits as a direct sum \[\Xi^{(\alpha)} \simeq \bigoplus_{W/\langle s \rangle}\Delta_w^{(\alpha)} * \Xi^{(\alpha)}_s\] indexed by $w \in W$ satisfying $w < ws$.
\end{proposition}
\begin{proof}
Proposition \ref{LocalizedExtensions} gives a splitting $\Xi^{(\alpha)} \simeq \bigoplus_{W/\langle s \rangle} K_w^{(\alpha)}$, where $K_w^{(\alpha)}$ lies in the subcategory generated under extensions and shifts by $\Delta_w^{(\alpha)}$ and $\Delta_{ws}^{(\alpha)}$.
By Proposition \ref{XiCoStandard} and Equation \eqref{StandardtoCostandard}, each summand admits a standard and costandard filtration \[0 \rightarrow \Delta_{ws}^{(\alpha)} \rightarrow K_w^{(\alpha)} \rightarrow \Delta_w^{(\alpha)} \rightarrow 0 \quad \text{and} \quad 0 \rightarrow \nabla_w^{(\alpha)} \rightarrow K_w^{(\alpha)} \rightarrow \nabla_{ws}^{(\alpha)} \rightarrow 0.\]
Proposition \ref{Key} says $\nabla_{w^{-1}}^{(w\alpha)} \simeq \Delta_{w^{-1}}^{(w\alpha)}$ is clean, so $\nabla_{w^{-1}}^{(w\alpha)} * K_w^{(\alpha)}$ admits standard and costandard filtrations as in \eqref{FilUnique}. This implies $K_w^{(\alpha)} \simeq \Delta_w^{(\alpha)}* \Xi_s^{(\alpha)}$ by Lemma \ref{XisUnique}.
 \end{proof}
 
 \subsection{Example}
Let $G = \PGL(3)$ with simple coroots $\alpha_1$ and $\alpha_2$. The non-simple positive coroot is $\beta \coloneqq \alpha_1 + \alpha_2$. Away from the walls $\check{T}_{\alpha_2}$ and $\check{T}_{\beta}$, the localized big tilting sheaf splits \[\Xi^{(\alpha_1)} \simeq (\Delta_1^{(\alpha_1)} \oplus \Delta_{s_2}^{(\alpha_1)} \oplus \Delta_{s_1s_2}^{(\alpha_1)}) * \Xi^{(\alpha_1)}_{s_1}.\]
The extension $\Delta_{s_1s_2}^{(\alpha_1)} \simeq \Delta_{s_1}^{(\beta)} * \Delta_{s_2}^{(\alpha_1)} \simeq \nabla_{s_1}^{(\beta)} * \nabla_{s_2}^{(\alpha_1)} \simeq \nabla_{s_1s_2}^{(\alpha_1)}$ is clean.

 \section{Localizing R-bimodules}
Here we separate the union of graphs into pairs by localizing away from all but one wall.
Let $\beta$ be a coroot, and $t \in W$ be the corresponding reflection. If $M \in \DBim(R)$, write $M^{(\beta)} \coloneqq M \otimes_R R^{(\beta)}$ for its right localization away from all walls except $\check{T}_{\beta}$.


\subsection{Uniformizing the dual torus}
The following lemma uses the assumption that $G$ has connected center. It is clear for classical groups because then the Weyl group permutes the entries of a diagonal torus. Below is a uniform proof using the affine Weyl group.

\begin{lemma}\label{AdjointWalls}
Assume $k = \bQ$ or $\bF_p$. If $\zeta \in \check{T}^{(\beta)}(\overline{k})$ is fixed by $w \in W$, then $\zeta \in \check{T}_{\beta}(\overline{k})$ is on the wall, and $w = t$ is corresponding reflection.
\end{lemma}
\begin{proof}
Choose an embedding $\overline{k}^{\times} \hookrightarrow \bC^{\times}$ and identify $\check{T}(\overline{k}) = \check{\Lambda} \otimes \overline{k}^{\times}$ with its image in the complex torus $\check{\Lambda} \otimes \bC^{\times}$. 
It suffices to replace $\zeta$ by $v \zeta$ and replace $\beta$ by $v\beta$ where $v \in W$. Therefore we may assume that $\zeta = e^{2 \pi i \check{X}}$ is the exponential of $\check{X} \in \check{\Lambda} \otimes \bC$ in the fundamental alcove. 
The fundamental alcove is bounded by the fixed loci of finite and affine simple reflections. 

\begin{enumerate}
\item[-] If $\check{X}$ is fixed by a finite simple reflection $s$, then also $s\zeta = \zeta$ is fixed. 
\item[-] If $\check{X}$ is fixed by the affine simple reflection $t_0 e^{-\check{\beta}_0} \in W^{\aff}$, then $\langle \beta_0, \check{X} \rangle = 1$ so $t_0 \zeta = \zeta$.
\end{enumerate}
Here $t_0 \in W$ is the reflection corresponding to the longest coroot $\beta_0$.

Let $\check{\Lambda}^{\rt} \subset \check{\Lambda}$ be the span of the roots. Since $G$ has connected center, $(\check{\Lambda}^{\rt} \otimes \bQ) \cap \check{\Lambda} = \check{\Lambda}^{\rt}$. The assumption $w\zeta = \zeta$ implies that $w\check{X} - \check{X} \in (\check{\Lambda}^{\rt} \otimes \bQ) \cap \check{\Lambda} = \check{\Lambda}^{\rt}$. Therefore $\check{X}$ is fixed by some affine Weyl group element $w e^{\check{\lambda}} \in W^{\aff} =  W \ltimes \check{\Lambda}^{\rt}$. Hence $\check{X}$ lies on the boundary of the fundamental alcove. Since $\zeta \in \check{T}$ lies on at most 1 wall, $\check{X}$ is fixed by exactly 1 nontrivial affine Weyl group element. Therefore $w = t$ is a reflection and $\zeta \in \check{T}_{\beta}$ is in the kernel of $\beta$.
\end{proof}

\subsection{Union of graphs}
Let $\Gamma_w^{(\beta)} \coloneqq \Spec R_w^{(\beta)}$ be the right localization of the graph of $w$.

\begin{proposition}\label{CoherentSplitting}
After right localizing away from all walls except $\check{T}_{\beta}$, \[R \otimes_{R^W} R^{(\beta)} \simeq \prod_{W/\langle t \rangle} R_w \otimes_{R^t} R^{(\beta)}.\]
\end{proposition}
\begin{proof}
It suffices to work over $k = \bQ$ or $\bF_p$. The closed subschemes $\Gamma_w^{(\beta)} \cup \Gamma_{wt}^{(\beta)} \subset \bigcup \Gamma_w^{(\beta)}$ are disjoint by Lemma \ref{AdjointWalls}, so they are separate connected components.
Lemma \ref{Reduced} implies
\[\Spec(R \otimes_{R^W} R^{(\beta)}) \simeq \bigcup_W \Gamma_w^{(\beta)} \simeq \bigsqcup_{W/\langle t \rangle} \Gamma_{w}^{(\beta)} \cup \Gamma_{wt}^{(\beta)} \simeq \bigsqcup_{W/ \langle t \rangle} \Spec(R_w \otimes_{R^t} R^{(\beta)}).\qedhere\]
\end{proof}

The following reducedness lemma allowed us to argue geometrically.

\begin{lemma}\label{Reduced}
The fiber product $\check{T} \times_{\check{T}/\!/W} \check{T} = \bigcup \Gamma_w$ is the union of graphs of Weyl group elements with the reduced induced scheme structure.
\end{lemma}
\begin{proof}
Both $\check{T} \times_{\check{T}/\!/W} \check{T}$ and $\bigcup \Gamma_w$ are closed subschemes of $\check{T} \times \check{T}$ with the same $\overline{k}$-points. Therefore it suffices to show that $\check{T} \times_{\check{T}/\!/W} \check{T}$ is reduced.
Indeed $R \otimes_{R^W} R$ is free as a right $R$-module, so it injects into its right localization \[R \otimes_{R^W} R \hookrightarrow R \otimes_{R^W}  \Frac(R) \simeq \prod_W  \Frac(R)_w.\qedhere\]\end{proof}

\section{Uncompleting Soergel's Endomorphismensatz}
Here we calculate endomorphisms of the big tilting sheaf. After localizing away from all but one wall, this reduces to a calculation in semisimple rank 1.

\subsection{The bimonodromy map}
Here we construct the map appearing in Theorem \ref{main} by repeating the proof of Proposition 6.4 of \cite{BR}.
\begin{proposition}
Bimonodromy factors through map of free right $R$-modules \beq \label{Bimon} R \otimes_{R^W} R \rightarrow \Hom(\Xi, \Xi).\eeq
\end{proposition}
\begin{proof}
Associated graded for the standard filtration induces an injection $\gr: \Hom(\Xi, \Xi) \hookrightarrow \prod R_w$ by Proposition \ref{Hom0}, as in Corollary 6.3 of \cite{BR}.
Therefore bimonodromy factors through the quotient $R \otimes R \rightarrow R \otimes_{R^W} R$.

The Pittie-Steinberg theorem \cite{St} implies $R \otimes_{R^W} R$ is free as a right $R$-module. Proposition \ref{XiCoStandard} and Equation \eqref{StandardtoCostandard} imply $\Hom(\Xi, \Xi)$ is a free right $R$-module concentrated in degree 0. 
\end{proof}

\subsection{Proof of Theorem \ref{main}}
If the determinant of \eqref{Bimon} was not invertible, it would vanish on a codimension 1 subvariety by Hartogs' lemma, see Theorem 38 of \cite{Mat}. The following lemma implies that \eqref{Bimon} is an isomorphism after localizing away from all but any one wall. The complement $\check{T} - \bigcup \check{T}^{(\beta)}$ is the locus where multiple walls intersect, which has codimension 2. Therefore Hartogs' lemma implies Soergel's Endomorphismensatz $R \otimes_{R^W}R \simeq \Hom(\Xi, \Xi)$.

\begin{lemma}\label{LocEndProp}
After localizing away from all walls except $\check{T}_{\beta}$, bimonodromy \eqref{Bimon} induces an isomorphism \beq \label{LocEnd} R \otimes_{R^W} R^{(\beta)} \simeq \Hom(\Xi, \Xi)^{(\beta)}.\eeq
\end{lemma}
\begin{proof}
First suppose $\beta = \alpha$ is a simple coroot. 
Propositions \ref{SimpleEnd} and \ref{VStandard}\eqref{VStandard3} imply that \[R_w \otimes_{R^s} R^{(\alpha)} \simeq R_w \otimes_R \bV(\Xi_s)^{(\alpha)} \overset{\eqref{IntertwiningIsomorphism}}{\simeq} \Hom(\Delta_w * \Xi, \Delta_w * \Xi_s)^{(\alpha)} \simeq \bV(\Delta_w * \Xi_s)^{(\alpha)} \overset{\eqref{HomTensor}}{\simeq} \bV(\Delta_w^{(\alpha)} * \Xi^{(\alpha)}_s).\] Propositions \ref{CoherentSplitting} and \ref{TiltingSplitting} imply that both sides of \eqref{LocEnd} split as a direct sum indexed by minimal length $s$ coset representatives. Therefore \[R \otimes_{R^W} R^{(\alpha)} \simeq \prod_{W/\langle s \rangle} R_w \otimes_{R^s} R^{(\alpha)} \simeq \bigoplus_{W/\langle s \rangle} \bV(\Delta_w^{(\alpha)} * \Xi^{(\alpha)}_s) \simeq \bV(\Xi^{(\alpha)}) \simeq \Hom(\Xi, \Xi)^{(\alpha)}.\] 

For an arbitrary coroot, write $\beta = w\alpha$ for $\alpha$ simple. Proposition \ref{VStandard} implies $\Xi^{(\beta)} \simeq \Xi^{(\alpha)} * \Delta_w^{(\beta)}$. So the simple coroot case implies $R \otimes_{R^W} R^{(\beta)} \simeq \Hom(\Xi, \Xi)^{(\beta)}$.
\end{proof}

\subsection{Recompleting Soergel's Endomorphismensatz}
The pro-unipotent Endomorphismensatz is proved in Proposition 4.7.3 of \cite{BY}, and extended to modular coeffients in \cite{BR}. Below is a short proof by formally completing Theorem \ref{main}.

Let $I \subset R$ be the ideal of functions vanishing at the identity in $\check{T}$. Consider the pro-unipotent tilting sheaf $\Xi^{\wedge} \coloneqq \varprojlim \Xi/I^n$ in the completed category of \cite{BY}.

\begin{corollary}\label{Complete} 
In the completed category $\Hom^0(\Xi^{\wedge}, \Xi^{\wedge}) \simeq R \otimes_{R^W} R^{\wedge}$.
\end{corollary}
\begin{proof}
If $m \geq n$ then $\Hom^0(\Delta_w/I^m, \nabla_v/I^n) \rightarrow \Hom^0(\Delta_w, \nabla_v/I^n)$ is an isomorphism in degree 0.
Therefore by the standard and costandard filtrations \[\Hom^0(\Xi/I^m, \nabla/I^n) \xrightarrow{\sim} \Hom^0(\Xi, \nabla/I^n) \quad \text{and hence} \quad \Hom^0(\Xi/I^m, \Xi/I^n) \xrightarrow{\sim} \Hom^0(\Xi, \Xi/I^n).\] 
Similarly $\Hom(\Xi, \Xi/I^n) \simeq \Hom(\Xi, \Xi)/I^n$. Theorem \ref{main} implies that in the completed category \begin{align*}\Hom^0(\Xi^{\wedge}, \Xi^{\wedge}) &\coloneqq \varprojlim_n \varinjlim_m \Hom^0(\Xi/I^m, \Xi/I^n) \\ &\simeq \varprojlim_n \Hom^0(\Xi, \Xi/I^n) \\ &\simeq \varprojlim_n R \otimes_{R^W} R/I^n \\ &\simeq R \otimes_{R^W} R^{\wedge}. \qedhere\end{align*}
\end{proof}

\section{Soergel bimodules and tilting sheaves}
Here we recall the additive category of multiplicative Soergel bimodules \cite{Eb}, and prove that its $\Hom$ spaces are free over $R$. Then we describe the universal monodromic Hecke category as the bounded homotopy category of Bott--Samelson tilting sheaves.

\subsection{The Bott--Samelson construction}
Let $\underline{x} = s_1 \dots s_r$ be an expression for $x \in W$ as a product of simple reflections. Define the Bott--Samelson bimodule and tilting sheaf \[B_{\underline{x}} \coloneqq R \otimes_{R^{s_1}} R \dots \otimes_{R^{s_r}} R \; \in \; \Bim(R) \quad \text{and} \quad \Xi_{\underline{x}} \coloneqq \Xi_{s_1}*\dots \Xi_{s_r} \; \in \; \DShv_{(B)}(Y).\] The proof of Proposition 7.8 of \cite{BR} shows that $\Xi_{\underline{x}}$ admits standard and costandard filtrations

Let $\SBim(R)$ be the full additive subcategory of finite sums and summands of $B_{\underline{x}}$.
Let $\Tilt_{(B)}(Y)$ be the full additive subcategory of finite sums and summands of $\Xi_{\underline{x}}$.
Soergel's functor \beq \label{VTilt} \bV \simeq \Hom(\Xi, -): \Tilt_{(B)}(Y) \rightarrow \SBim(R), \qquad \Xi_{\underline{x}} \mapsto B_{\underline{x}},\eeq
sends Bott--Samelson tilting sheaves to Bott--Samelson bimodules by Propositions \ref{Monoidal} and \ref{SimpleEnd}.

\subsection{Freeness}
To later invoke Hartogs' lemma, we need the following freeness.

It is easy to see that $B_{\underline{x}}$ admits a filtration with graded pieces $R_w$. Below we use a geometric argument to show that all graded pieces $R_1$ can be arranged to appear last in the filtration. This is essential so that 0 appears before any $\Ext^1$ terms in \eqref{BSLES}.

\begin{proposition}\label{SBimFree}
For $\underline{x}$ and $\underline{z}$ two expressions, $\Hom^0(B_{\underline{x}}, B_{\underline{z}})$ is free as a right $R$-module.
\end{proposition}
\begin{proof}
Using the self-adjunction from Lemma \ref{SelfAdjoint}, it suffices to consider the case $\underline{z} = 1$.

The kernel of $\Xi_{\underline{x}} \rightarrow \Xi_{\underline{x}}|_{Y_1} \simeq \Delta_1^{\oplus n}$ admits a standard filtration with graded pieces $\Delta_w$ indexed by $w \neq 1$. Applying Soergel's functor $\bV$ gives a short exact sequence of bimodules \[0 \rightarrow \ker \rightarrow B_{\underline{x}} \rightarrow R_1^{\oplus n} \rightarrow 0\] such that the kernel is filtered with graded pieces $R_w$ indexed by $w \neq 1$, by Equation \eqref{VTilt} and Proposition \ref{VStandard}. For such $w \neq 1$ we have $\Hom^0(R_w, R_1) \simeq 0$. Therefore $\Hom^0(\ker, R_1) \simeq 0$.

There is a long exact sequence \beq \label{BSLES} 0 \rightarrow \Hom^0(R_1^{\oplus n}, R_1) \rightarrow \Hom^0(B_{\underline{x}}, R_1) \rightarrow \Hom^0(\ker, R_1) \simeq 0 \rightarrow \dots \eeq so $\Hom^0(B_{\underline{x}}, R_1) \simeq R_1^{\oplus n}$ is free as a right $R$-module.
\end{proof}

We used the following self-adjunction, a multiplicative version of Proposition 5.10 of \cite{So07}.

\begin{lemma}\label{SelfAdjoint}
The functor \beq \label{SelfAdjointMap} R \otimes_{R^s} -:\Mod(R) \rightarrow \Mod(R)\eeq is self-adjoint.
\end{lemma}
\begin{proof}
Since $G$ has connected center, there is a coweight $\omega \in \Lambda$ satisfying $\langle \omega, \check{\alpha} \rangle = 1$. Define the Demazure operator \[D: R \rightarrow R^s, \qquad f \mapsto (f - sf)/(e^{\omega} - e^{s\omega}),\] using that $e^{\omega} - e^{s\omega}$ generates the ideal of functions vanishing on $\check{T}_{\alpha}$. Hence \[R = R^s \oplus e^{\omega} R^s, \qquad f \; \leftrightarrow \; (D(e^{s\omega}f), \; e^{\omega}D(f))\] is a free $R^s$-module of rank 2. 

Let $1^*, (e^{\omega})^* \in \Hom_{R^s}(R, R^s)$ be the dual $R^s$-basis to $1, e^{\omega} \in R$. The $R$-linear map \[R \xrightarrow{\sim} \Hom_{R^s}(R, R^s), \qquad 1 \; \mapsto \; 1^*, \quad  e^{\omega} \; \mapsto \; e^{\omega} 1^* =  (e^{\omega} + e^{s\omega})1^* - e^{\omega + s\omega} (e^{\omega})^*\] is an isomorphism because, as an $R^s$-linear map, it has unit determinant $-e^{\omega + s\omega} \in R^s$. Therefore \[R \otimes_{R^s} - \simeq \Hom_{R^s}(R, -):\Mod(R^s) \rightarrow \Mod(R)\] is both left and right adjoint to restriction, and hence \eqref{SelfAdjointMap} is self-adjoint.
\end{proof}

\subsection{Homotopy category of tilting sheaves} Let $\KTilt_{(B)}(Y)$ denote the bounded homotopy category of $\Tilt_{(B)}(Y)$. Following Proposition 1.5 of \cite{BBM}, this recovers the universal monodromic Hecke category.  

\begin{proposition}\label{TiltShv}
There is an equivalence $\DShv_{(B)}(Y) \simeq \KTilt_{(B)}(Y)$.
\end{proposition}
\begin{proof}
Equation \eqref{StandardtoCostandard} implies that $\Hom(\Xi_{\underline{x}}, \Xi_{\underline{z}})$ is concentrated in degree 0. Therefore the realization functor $\KTilt_{(B)}(Y) \rightarrow \DShv_{(B)}(Y)$ constructed in \cite{B06} is fully faithful.

For each stratum $Y_w$, choosing a reduced expression for $w$ gives a Bott--Samelson tilting sheaf $\Xi_{\underline{w}}$ whose support is the closure $\overline{Y}_w$. By induction on length, the essential image contains all $\Delta_w$. Hence using Lemma \ref{Compact} the realization functor is essentially surjective.
\end{proof}

\section{Uncompleting Soergel's Struktursatz and BGS Koszul duality}
Here we show that Soergel's functor is fully faithful on Bott--Samelson tilting sheaves, and deduce that $\DShv_{(B)}(Y) \simeq \KSBim(R)$ is the bounded homotopy category of multiplicative Soergel bimodules. This implies Eberhardt's conjecture that uncompletes Koszul duality.

The usual proof of Soergel's Struktursatz uses a certain socle and cosocle calculation, Lemma 2.1 of \cite{BBM}, so we instead argue by reducing to semisimple rank 1.

\subsection{Localizing Bott--Samelson sheaves}
Here we explain the splitting of Bott--Samelson tilting sheaves after localizing away from all walls except $\check{T}_{\beta}$. We reduce to the case of a simple coroot by choosing $w \in W$ such that $\ell(w) < \ell(wt)$ and $w^{-1} \beta = \alpha$ is simple. 

\begin{lemma}\label{BSSheafLoc}
There is a splitting of $\Xi_{\underline{x}}^{(\beta)} * \Delta_w^{(\alpha)}$ with summands of the form $\Delta_v^{(\alpha)}$ and $\Delta_v^{(\alpha)} * \Xi_s^{(\alpha)}$. Here $v \in W$ are such that $\Delta_v^{(\alpha)} \simeq \nabla_v^{(\alpha)}$ is clean.\end{lemma}
\begin{proof}
Let $s_1$ be first simple reflection in the expression $\underline{x}$, and $\alpha_1$ be the corresponding simple coroot. Then $\Xi_{\underline{x}} \simeq \Xi_{s_1} * \Xi_{\underline{z}}$ where $\underline{x} = s_1 \underline{z}$. By induction on the length of the expression, $\Xi_{\underline{z}}^{(\beta)} * \Delta_w^{(\alpha)}$ splits with summands of the desired form. Therefore it suffices to show that $\Xi_{s_1}^{(v\beta)} * \Delta_v^{(\alpha)}$ and $\Xi_{s_1}^{(v\alpha)} * \Delta_v^{(\alpha)} * \Xi_s^{(\alpha)}$ both split with summands of the desired form.

If $\alpha_1 \neq v\alpha$ then $\Xi_{s_1}^{(v\alpha)} \simeq \Delta_1^{(v\alpha)} \oplus \Delta_{s_1}^{(v\alpha)}$ splits. Either $v < s_1v$ or $v > s_1v$ but, since $\Delta_{s_1}^{(v\alpha)} \simeq \nabla_{s_1}^{(v\alpha)}$ is clean, in both cases \[\Xi_{s_1}^{(v\alpha)} * \Delta_v^{(\alpha)} \simeq \Delta_v^{(\alpha)} \oplus \Delta_{s_1v}^{(\alpha)}\] and $\Delta_{s_1v}^{(\alpha)} \simeq \Delta_{s_1}^{(v\alpha)} * \Delta_v^{(\alpha)}$ is clean. Therefore \[\Xi_{s_1}^{(v\alpha)} * \Delta_v^{(\alpha)} * \Xi_s^{(\alpha)} \simeq (\Delta_v^{(\alpha)} * \Xi_s^{(\alpha)}) \oplus (\Delta_{s_1v}^{(\alpha)} * \Xi_s^{(\alpha)})\] splits with summands of the desired form.

If $\alpha_1 = v\alpha$ then \[\Xi_{s_1}^{(\alpha_1)} * \Delta_v^{(\alpha)} \simeq \Delta_v^{(\alpha)} * \Xi_s^{(\alpha)}.\] Indeed since $\Delta_v^{(\alpha)}$ is clean, both sides admit standard and costandard filtrations with graded pieces indexed by $v$ and $vs$. Therefore Proposition \ref{Rank1Square} implies that \[\Xi_{s_1}^{(\alpha_1)} * \Delta_v^{(\alpha)} * \Xi_s^{(\alpha)}  \simeq \Delta_v^{(\alpha)} * \Xi_s^{(\alpha)} * \Xi_s^{(\alpha)}  \simeq (\Delta_v^{(\alpha)} * \Xi_s^{(\alpha)})^{\oplus 2}\] splits with summands of the desired form.
\end{proof}

\subsection{Uncompleting Soergel's  Struktursatz}
Soergel's Struktursatz says that $\bV$ is fully faithful on Bott--Samelson tilting sheaves.

\begin{theorem} \label{Struktursatz}
Soergel's functor $\bV \simeq \Hom(\Xi, -)$  induces isomorphisms \beq \label{FullyFaithful}\Hom^0_{\DShv_{(B)}(Y)}(\Xi_{\underline{x}}, \Xi_{\underline{z}}) \rightarrow \Hom^0_{\Bim(R)}(B_{\underline{x}}, B_{\underline{z}}).\eeq
\end{theorem}
\begin{proof}
Equation \eqref{StandardtoCostandard} and Proposition \ref{SBimFree} imply that both sides of \eqref{FullyFaithful} are free right $R$-modules. 
If $\ell(w) < \ell(wt)$ and $w^{-1} \beta = \alpha$ is a simple coroot then $\Xi_{\underline{x}}^{(\beta)} * \Delta_w^{(\alpha)}$ and $\Xi_{\underline{z}}^{(\beta)} * \Delta_w^{(\alpha)}$ split with summands of the form $\Delta_v^{(\alpha)}$ and $\Delta_v^{(\alpha)} * \Xi_s^{(\alpha)}$ by Lemma \ref{BSSheafLoc}. Propositions \ref{XiCoStandard} and \ref{SimpleEnd} imply \[\bV(\Delta_v^{(\alpha)}) \simeq R_v^{(\alpha)} \quad \text{and} \quad \bV(\Delta_v^{(\alpha)} * \Xi_s^{(\alpha)}) \simeq R_v \otimes_{R^s} R^{(\alpha)}.\] By Hartogs' lemma it suffices to prove $\bV$ is fully faithful on such summands. Indeed the only nonzero terms are obtained from Lemma \ref{Rank1Struk} by convolving both arguments by the same $\Delta_v^{(\alpha)}$.
\end{proof}

By localizing we reduced to the following calculations in semisimple rank 1.

\begin{lemma}\label{Rank1Struk}
Soergel's functor $\bV$ induces isomorphisms 
\begin{enumerate} 
\item \label{XiXi}$\Hom^0(\Xi_s, \Xi_s) \rightarrow \Hom^0(R \otimes_{R^s} R, R \otimes_{R^s} R)$,
\item \label{DeltaDelta}$\Hom^0(\Delta_1, \Delta_1) \rightarrow \Hom^0(R, R)$,
\item \label{DeltaXi} $\Hom^0(\Delta_1, \Xi_s) \rightarrow \Hom^0(R, R \otimes_{R^s} R)$,
\item \label{XiDelta}$\Hom^0(\Xi_s, \Delta_1) \rightarrow \Hom^0(R \otimes_{R^s} R, R)$.
\end{enumerate}
\end{lemma}
\begin{proof}
Both sides of \eqref{XiXi} are $R \otimes_{R^s} R$ and the identity map goes to the identity, therefore it is an isomorphism. Similarly \eqref{DeltaDelta} is an isomorphism.

By Proposition \ref{SimpleEnd}, Soergel's functor sends \[0 \rightarrow \nabla_1 \rightarrow \Xi_s \rightarrow \nabla_s \rightarrow 0 \quad \text{ to } \quad 0 \rightarrow R_1 \rightarrow R \otimes_{R^s} R \rightarrow R_s \rightarrow 0.\] 
Since $\Hom(\Delta_1, \nabla_s) \simeq 0$ and $\Hom^0(R_1, R_s) \simeq 0$, the vertical maps are isomorphisms in
\[\begin{tikzcd}[cramped]
\Hom(\Delta_1, \Xi_s) \arrow[r, "\eqref{DeltaXi}"] & \Hom(R, R \otimes_{R^s} R)  \\
\Hom(\Delta_1, \nabla_1) \arrow[r, "\sim"'] \arrow[u, "\sim"] & \Hom(R_1, R_1) \arrow[u, "\sim"'].
\end{tikzcd}\]
Therefore \eqref{DeltaXi} is an isomorphism. Similarly \eqref{XiDelta} is an isomorphism.
\end{proof}

\subsection{Proof of Theorem \ref{Koszul}}
Theorem \ref{Struktursatz} gives an equivalence of additive categories $\Tilt_{(B)}(Y) \simeq \SBim(R)$ between tilting sheaves and multiplicative Soergel bimodules. Taking bounded homotopy categories implies universal ungraded Koszul duality \[\DShv_{(B)}(Y) \simeq \KTilt_{(B)}(Y) \simeq \KSBim(R) \simeq \DK_{\check{B}}(\check{X}),\] by \cite{Eb} and Proposition \ref{TiltShv}.

\subsection{Remark on quantum parameters}
In quantum K-theoretic geometric Satake \cite{Elias, CK}, the quantum parameter arises from loop rotation equivariance. In quantum geometric Langlands \cite{G}, the quantum parameter arises from monodromy about the central extension line bundle. The extension of universal Koszul duality to Kac--Moody groups in \cite{EE} exchanges loop rotation equivariance for central extension monodromy.

\appendix

\section{Groups with disconnected center}\label{ApendixEnd}
If $G$ has disconnected center then Lemma \ref{AdjointWalls} may fail, so Theorem \ref{main} must be modified as follows.
Choose a finite central subgroup $Z \subset G$ such that $G^{\ad} \coloneqq G/Z$ has connected center. There exists a reductive group $\check{G}^{\sct}$ with simply connected derived subgroup, and a finite central subgroup $\check{Z} \subset \check{G}^{\sct}$ Cartier dual to $Z$, such that $\check{G}^{\sct}/\check{Z} = \check{G}$.

Let $R'$ be the group ring of the coweight lattice of $T/Z$. The left and right torus action factors through the antidiagonal quotient $T \times T \rightarrow (T \times T)/Z \curvearrowright Y$. Therefore the $R \otimes R$-linear structure on $\DShv_{(B)}(Y)$ extends to an $(R' \otimes R')^{\check{Z}}$-linear structure.


We assumed $G$ had connected center to simplify notation. In general the same arguments show \[(R' \otimes_{(R')^{W}} R')^{\check{Z}} \simeq \Hom(\Xi, \Xi).\]

\subsection{Example}
Let $G = \SL(2)$ and $\alpha$ be the simple coroot. Then $R\otimes_{R^W} R$ equals functions on $\Gamma_1 \cup \Gamma_s$, the union of the graphs of the Weyl group elements. The graphs meet at two points, \[\big(\begin{smallmatrix} 1 & \\ & 1 \end{smallmatrix}\big) \quad \text{and} \quad \big(\begin{smallmatrix} -1 & \\ & 1 \end{smallmatrix}\big).\] Replacing $R\otimes_{R^W}R$ by $(R'\otimes_{(R')^{W}}R')^{\pi_1(\check{G})}$ has the effect of separating the graphs so they only intersect at the identity.

The tilting sheaf admits a costandard filtration \[0 \rightarrow \nabla_1 \rightarrow \Xi \rightarrow \nabla_s \rightarrow 0 \quad \text{classified by} \quad 1 \in \Ext^1(\nabla_s, \nabla_1) \simeq k_1,\]  the augmentation module at $1 \in \check{T}$. After localizing $\Xi[(e^{\alpha} - 1)^{-1}]$ splits with endomorphisms \[\Hom(\Xi, \Xi)[(e^{\alpha} - 1)^{-1}] \simeq (R_1 \oplus R_s)[(e^{\alpha} - 1)^{-1}],\] so there is no second intersection.

\end{document}